\pdfoutput=1
\documentclass{amsart}
\usepackage{amsmath}
\usepackage{amssymb}
\usepackage{graphicx, color}

\theoremstyle{definition}
\newtheorem{definition}{Definition}

\theoremstyle{plain}
\newtheorem{theorem}{Theorem}
\newtheorem{lemma}{Lemma}
\newtheorem{corollary}{Corollary}
\newtheorem{proposition}{Proposition}
\newtheorem{remark}{Remark}

\begin{document}

\title[Strong and weak (1, 2) homotopies and new invariants]{Strong and weak (1, 2) homotopies on knot projections and new invariants}
\author[N. Ito]{Noboru Ito}
\author[Y. Takimura]{Yusuke Takimura}
\address{Waseda Institute for Advanced Study, 1-6-1 Nishi-Waseda Shinjuku-ku Tokyo 169-8050 Japan}
\email{noboru@moegi.waseda.jp}
\address{(Current address) Graduate School of Mathematical Sciences, The University of Tokyo, 3-8-1 Komaba Meguro-ku Tokyo 153-8914 Japan}
\email{noboru@ms.u-tokyo.ac.jp}
\address{Gakushuin Boy's Junior High School, 1-5-1 Mejiro Toshima-ku Tokyo 171-0031 Japan}
\email{Yusuke.Takimura@gakushuin.ac.jp}
\maketitle

\begin{abstract}
Every second flat Reidemeister move of knot projections
can be decomposed into two types thorough an inverse or direct self-tangency modification, respectively called strong or weak, when  orientations of the 
knot projections are arbitrarily provided.  Further, we introduce the notions of strong and weak (1, 2) homotopies; we define that two 
knot projections are 
strongly (resp.~weakly) (1, 2) homotopic
if and only if two 
knot projections are related by a finite sequence of first and strong (resp.~weak) second flat Reidemeister moves.    
This paper gives a new necessary and sufficient condition that two 
knot projections are 
not strongly (1, 2) homotopic.    
Similarly, we obtain a new necessary and sufficient condition in the weak (1, 2) homotopy case.  We also define a new integer-valued strong (1, 2) homotopy invariant.  Using it, we show that the set of the non-trivial prime 
knot projections
without $1$-gons 
 that can be trivialized under strong (1, 2) homotopy is disjoint from that of weak (1, 2) homotopy.    
We also investigate topological properties of the new invariant and give its generalization, a comparison of our invariants and Arnold invariants, and a table of invariants.    
\end{abstract}

\section{Introduction}
A {\it{knot projection}} is defined as the image of a generic immersion of a circle into a $2$-dimensional sphere.  
Historically, the equivalence classes of knot projections 
generated by the first and second flat Reidemeister moves shown in Fig.~\ref{y4t} are determined by Khovanov \cite[Theorem 2.2]{khovanov} via the notion of doodles, introduced by Fenn and Taylor \cite{FennTaylor, Fenn} (cf.~\cite[Theorem 2.2]{IT}).  Essentially, by using 
{\it{(1, 2)-reduced}} knot projections, which are 
knot projections without any $1$- or $2$-gons, we can detect whether two given 
knot projections are equivalent under the equivalence relations obtained by the first and the second flat Reidemeister moves.  
Here, a $1$-{\it{gon}} (resp.~$2$-{\it{gon}}) is the boundary of a disk with exactly one (resp.~two) vertex and exactly one (resp.~two) edges, where a knot projection consists of vertexes and edges (Fig.~\ref{y4t}).  
\begin{figure}[htbp]
\includegraphics[width=5cm]{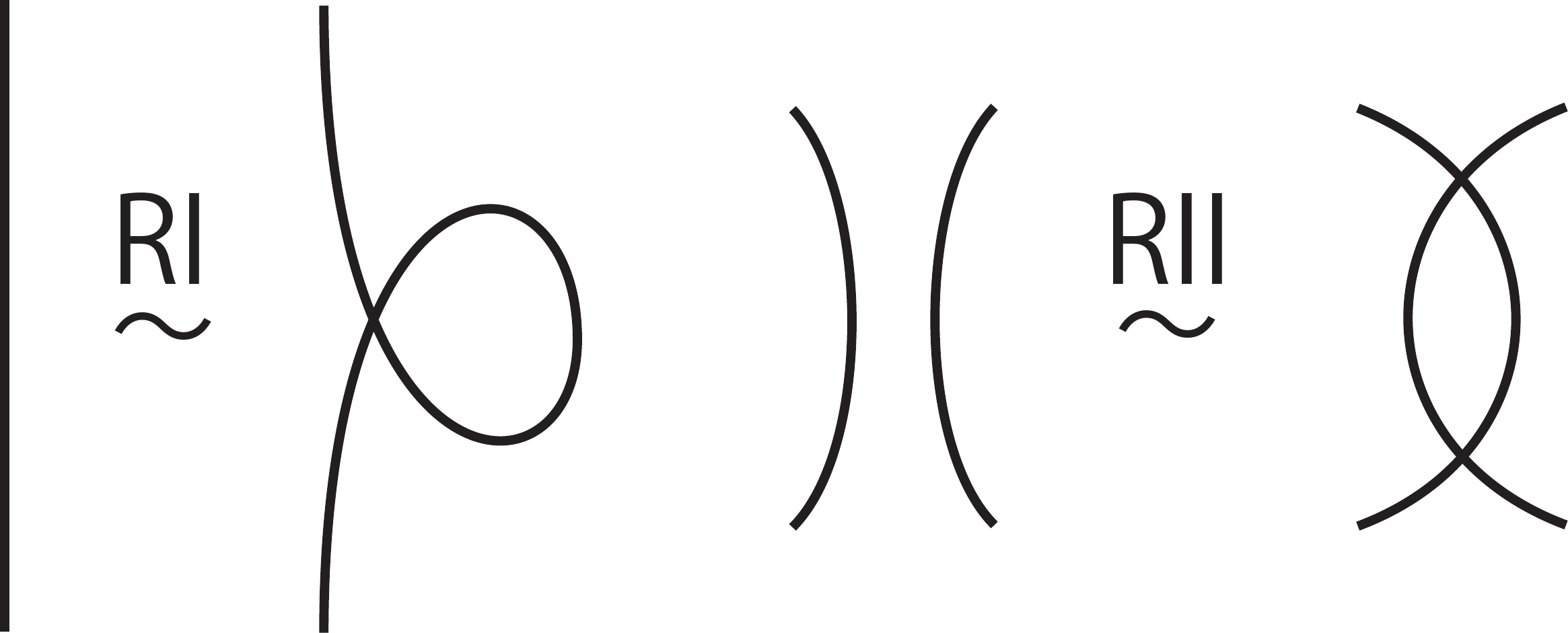}
\caption{The first and second flat Reidemeister moves for 
knot projections.}\label{y4t}
\end{figure}

Nonetheless, to the best of our knowledge, there have still been two open problems: find a nice necessary and sufficient condition for two knot projections being equivalent under a given equivalence relation, called {\it{strong $(1, 2)$ homotopy}} (resp.~{\it{weak $(1, 2)$ homotopy}}), consisting of the first flat Reidemeister moves denoted by  
RIs and the 
{\it{strong}} RI\!Is
 (resp.~
{\it{weak}} RI\!Is
) as defined by Fig.~\ref{5t}~
(center)
 (resp.~Fig.~\ref{5t}~
(right)).  
The equivalence induced by RI is denoted by $\stackrel{1}{\sim}$.  The equivalence induced by strong RI\!I (resp.~weak RI\!I) is denoted by $\stackrel{s2}{\sim}$ (resp.~$\stackrel{w2}{\sim}$).   
The $2$-gon appearing in a strong (resp.~weak) 
RI\!I is called a {\it{coherent}} (
resp.~{\it{incoherent
}}) $2$-gon if the $2$-gon is  (resp.~is not) be oriented by orienting a 
knot projection.  
\begin{figure}[htbp]
\includegraphics[width=10cm]{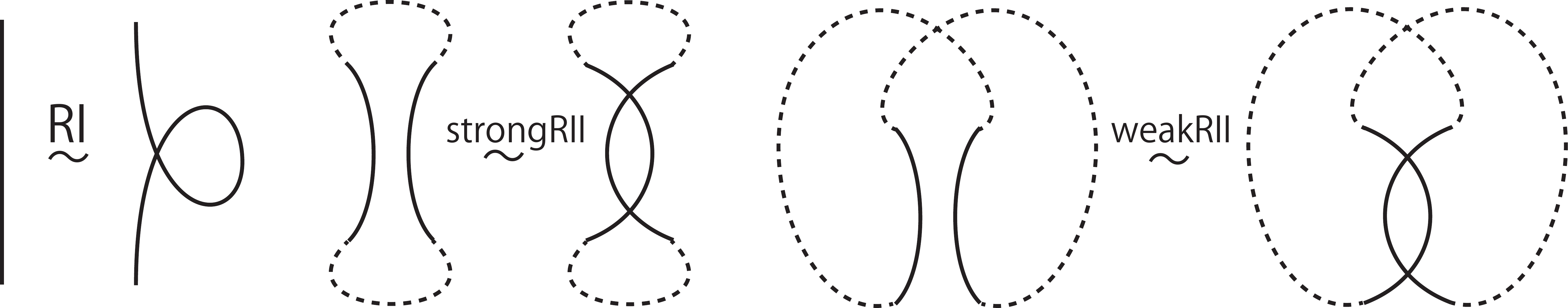}
\caption{
Strong RI\!I and weak RI\!I.  
Dotted arcs stand for the connection of the branches of two double points, which we focused on.}\label{5t}
\end{figure}
In this paper, let $\sim_s$ (resp.~$\sim_w$) be strong (resp.~weak) (1, 2) homotopy and let $\simeq$ be sphere isotopy (often simply called isotopy).   
The definition of the strong or weak RI\!I is a natural notion as a second flat Reidemeister move (Fig.~\ref{y4t}, right) is decomposed into just the two kinds of local moves shown in Fig.~\ref{5t}; one is strong RI\!I, and the other is a weak RI\!I.  The strong (resp.~weak) RI\!I is also treated as the inverse (resp.~direct) self-tangency perestroika as in Arnold \cite{arnold} if the knot projections are oriented.  In the rest of this paper, a knot projection having no double points is simply called a simple closed curve or a trivial knot projection.  

The answers to the two aforementioned problems appear in this paper (Theorem \ref{main1}), serving as the starting point.  
\begin{theorem}\label{main1}
Let a knot projection with no $1$-gons and no incoherent (resp.~coherent) $2$-gons be called a weak (resp.~strong) reduced knot projection $P_i^{wr}$ (resp.~$P_i^{sr}$), only decreasing double points by any 
RIs or weak (resp.~strong) RI\!Is
from a knot projection $P_i$.  

\noindent$(1)$ $P_1$ and $P_2$ are weakly (1, 2) homotopic if and only if $P_1^{wr}$ and $P_2^{wr}$ are isotopic.\label{claim1w} 

\noindent$(2)$ $P_1$ and $P_2$ are strongly (1, 2) homotopic if and only if $P_1^{sr}$ and $P_2^{sr}$ are isotopic.\label{claim1s}
\end{theorem}


Here, the next problem arises: how do the two equivalence relations, weak and strong (1, 2) homotopies, classify knot projections?  This paper obtains a partial answer to that problem.  
\begin{theorem}\label{main3}
Let $O$ be a knot projection having no double points, namely a simple circle on the $2$-sphere.  Then a knot projection $P$ is strongly (1, 2) homotopic and weakly (1, 2) homotopic to $O$ if and only if $P$ and $O$ are transformed into each other by RIs and isotopies.  
\end{theorem}
To prove Theorem \ref{main3}, this paper introduces a topological invariant of 
knot projections obtained by applying the replacement shown in Fig.~\ref{2} at every double point.  The replacements do not depend on an  orientation of a  
knot projection and are denoted as ``$A^{-1}$'' in \cite{IS}.  We call  replacement $A^{-1}$, or {\it{non-Seifert resolution}}.  
\begin{figure}[htbp]
\includegraphics[width=5cm]{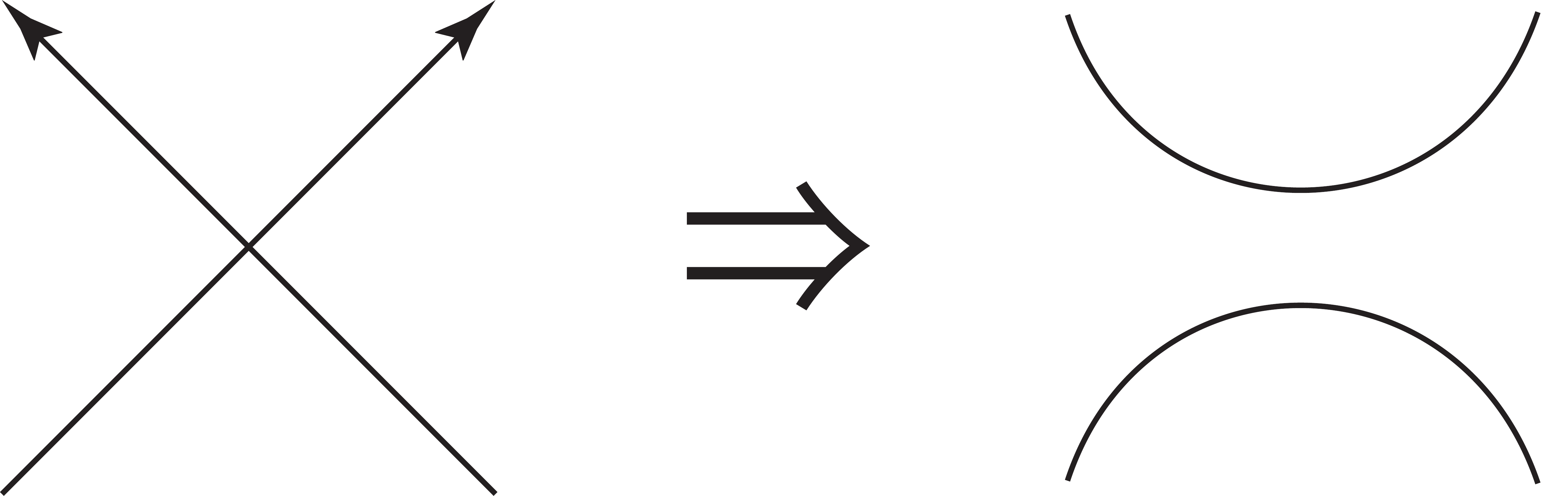}
\caption{The local replacement around a double point ($A^{-1}$).}\label{2}
\end{figure}
One may feel that this local replacement is similar to Seifert resolution (Fig.~\ref{seifert}).  
\begin{figure}[htbp]
\includegraphics[width=5cm]{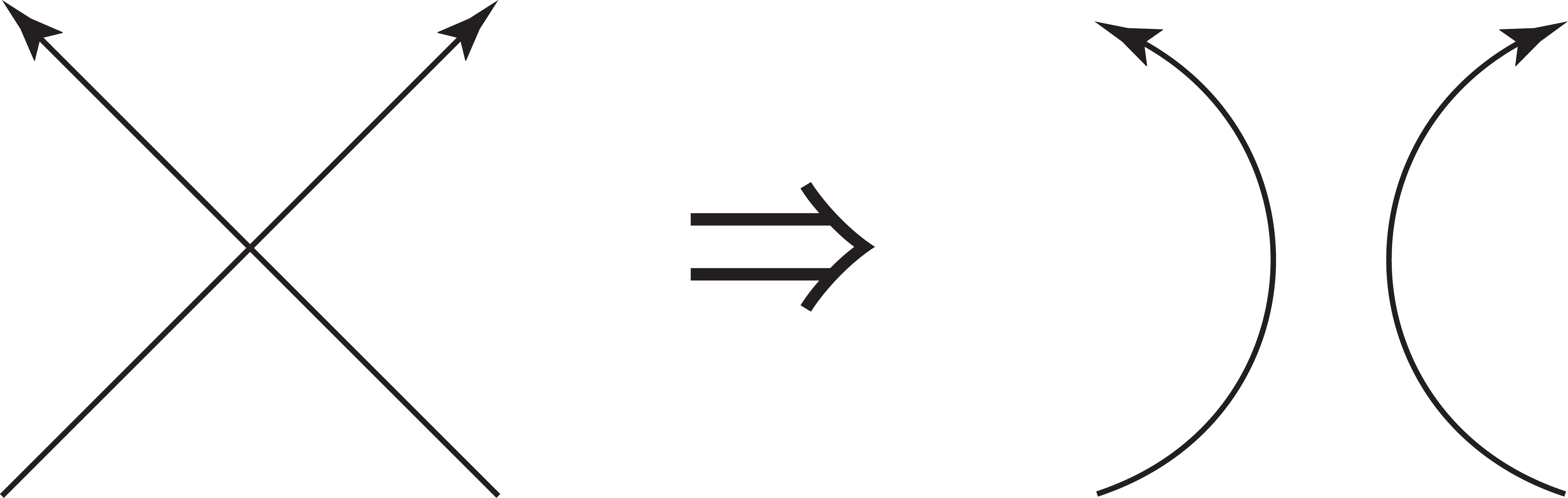}
\caption{Seifert resolution.}\label{seifert}
\end{figure}
Seifert resolution has many crucial roles in the basics of today's Knot Theory; for instance, genus,  Alexander polynomial, or Khovanov-Lee homology.  It also has an important role in the theory of 
generic immersed plane curves: e.g., it provides the Alexander numbering and Arnold invariants of plane  curves \cite{viro, shumakovitch}.  In these roles, one of the advantages of Seifert resolution is that it preserves the orientation of an oriented curve.  In comparison, $A^{-1}$ does not preserve the orientation if the curve is oriented, and thus one may think that $A^{-1}$ would not have such an important role. 

However, this paper shows that  $A^{-1}$ has nice properties and applications.  One of the advantages of $A^{-1}$ is that for a non-oriented curve, $A^{-1}$ does not change under 
RI, while each Seifert resolution changes the number of components under 
RI.  In fact, $A^{-1}$ gives invariants of 
knot projections under both 
RI and strong RI\!I as shown in Fig.~\ref{5t}.

For example, let us consider our familiar object, the chord diagram (often called Gauss diagram), which is one circle with finitely many chords where each chord connects the preimages of each double point of a knot projection.  The number $X(P)$, introduced by \cite{ITT}, is defined as the number of sub-chords as ``$\otimes$'' embedded in the whole chord diagram of a given knot projection $P$; $X(P)$ modulo $2$ is then an invariant under strong (1, 2) homotopy, though it is a $\mathbb{Z}/2\mathbb{Z}$-valued invariant.

An easily calculated $\mathbb{Z}$-valued strong (1, 2) homotopy invariant, called {\it{circle number}}, is introduced in this paper using non-Seifert resolution $A^{-1}$.  We also give a generalization of the $\mathbb{Z}$-valued invariant, called {\it{circle arrangement}}, under strong (1, 2) homotopy and obtain its table of prime knot projections with small numbers of double points.  We call these new invariants as a whole {\it{circle invariants}}.  

This paper is constructed as follows.  
Sec.~\ref{sec2.8} proves Theorem \ref{main1}.  Sec.~\ref{sec3} defines a new $\mathbb{Z}$-valued invariant, {\it{circle numbers}}, under strong (1, 2) homotopy and also gives a new, strictly stronger invariant, {\it{circle arrangements}}.  
Sec.~\ref{sec4} explores the characteristics of circle numbers.  Sec.~\ref{sec4.3} proves Theorem \ref{main3}.  Sec.~\ref{sec5} explores the characteristics of circle arrangements.  Finally, Sec.~\ref{sec6} obtains a table of prime knot projections with small numbers of double points including information on circle numbers and circle arrangements.

\section{Proof of Theorem \ref{main1}}\label{sec2.8}
We need two lemmas which are needed to prove Theorem~\ref{main1}.  
\begin{lemma}\label{lem1w}
A move RI increasing (resp.~decreasing) double points is denoted by $1a$ (resp.~$1b$).  
Weak RI\!I increasing (resp.~decreasing) double points is denoted by $w2a$ (resp.~$w2b$).  Any finite sequence generated by  
RIs and weak RI\!Is from 
a weak reduced knot projection to 
a knot projection can be replaced with a sequence of moves only of types $1a$ and $w2a$.  
\end{lemma}
\begin{lemma}\label{lem1s}
A move RI increasing (resp.~decreasing) double points is denoted by $1a$ (resp.~$1b$).  
Strong RI\!I increasing (resp.~decreasing) double points is denoted by $s2a$ (resp.~$s2b$).  Any finite sequence generated by 
RIs and strong RI\!Is from 
a strong reduced knot projection to 
a knot projection can be replaced with a sequence of moves only of types $1a$ and $s2a$.  
\end{lemma}
\noindent {\it{Proofs of Lemma~\ref{lem1w}} and Lemma~\ref{lem1s}}.  We can prove Lemmas \ref{lem1w} and \ref{lem1s} by restricting the general second flat Reidemeister move into either a 
weak RI\!I
 (for claim (1)) or 
a strong RI\!I
  (for claim (2)) in \cite[Proof of Theorem 2.2 (c)]{IT}.  The strategy of the two proofs of (1) and (2) shows up in Table~\ref{table1}.  In Table \ref{table1}, ``same'' means the same argument as that of \cite[Theorem 2.2 (c)]{IT}, $\emptyset$ means that it cannot occur, and the term ``restricted into strong RI\!I (resp.~weak RI\!I)'' means that the same argument is used, though considering the second flat Reidemeister move restricted to only a 
strong RI\!I (resp.~
weak RI\!I).  The proof \cite[Theorem 2.2 (c)]{IT} consists of Case 1--Case 4 by the last two moves when, in the sequence of the first and second flat Reidemeister moves, the first appearance of moves decreasing the number of double points occurs.  Each of Case 1--Case 4 has that (i): the last two moves are close and (ii): the last two moves are apart from one another.  Thus, using the proof of \cite[Theorem 2.2 (c)]{IT}, we have proofs of Lemmas \ref{lem1w} and \ref{lem1s}.  \hspace{\fill}$\Box$
\begin{table}
\begin{tabular}{|c|c|c|}\hline
&Weak (1, 2) &Strong (1, 2) \\ \hline
Case 1-(i) & same& same\\ \hline
Case 1-(ii) & same & same\\ \hline
Case 2-(i) & $\emptyset$ & same \\ \hline
Case 2-(ii) & restricted into weak RI\!I & restricted into strong RI\!I \\ \hline
Case 3-(i) & $\emptyset$ & same \\ \hline
Case 3-(ii) & restricted into weak RI\!I & restricted into strong RI\!I \\ \hline
Case 4-(i)  & restricted into weak RI\!I & restricted into strong RI\!I \\ \hline
Case 4-(ii) & restricted into weak RI\!I & restricted into strong RI\!I \\ \hline
\end{tabular}
\caption{Strategy to prove Lemmas \ref{lem1w} and \ref{lem1s}.}\label{table1}
\end{table}


\noindent{\it{Proof of Theorem~\ref{main1}}}.  
Assume that $P_1^{wr} \not\simeq P_2^{wr}$.  If $P_1 \sim_w P_2$, then $P_1^{wr} \sim_w P_2^{wr}$.  By Lemma \ref{lem1w}, there exists a finite sequence of moves only of types $1a$ and $w2a$ from $P_1^{wr}$ to $P_2^{wr}$.  Thus $P_2^{wr}$ has at least one $1$- or incoherent $2$-gon, which contradicts the definition of $P_2^{wr}$.  Therefore, if $P_1 \sim_w P_2$, $P_1^{wr} \simeq P_2^{wr}$.  If $P_1^{wr} \simeq P_2^{wr}$, by the definitions of $P_1^{wr}$ and $P_2^{wr}$, $P_1 \sim_w P_2$.  This completes the proof of claim (1).  

For claim (2) of Theorem \ref{main1}, a very similar proof to the above is established thorough appropriate replacements (e.g., $w \to s$, ``incoherent'' $\to$ ``coherent'').  That complete the proof of Theorem \ref{main1}.  
\hspace{\fill}$\Box$
\begin{corollary}
For an arbitrary knot projection $P$, there exists a unique knot projection $P^{wr}$ 
realizing the minimal number of 
double points up to weak (1, 2) homotopy.  
\end{corollary}
\begin{corollary}
For an arbitrary knot projection $P$, there exists a unique knot projection $P^{sr}$ 
realizing the minimum number of double points up to strong (1, 2) homotopy.  
\end{corollary}
\section{Definition of circle invariants}\label{sec3}
When an oriented knot projection is given, 
let us consider the local replacement of Fig.~\ref{2} from the left figure to the right figure for every double point.  By this definition, this local replacement does not depend on an orientation of 
a knot projection.  This local replacement is introduced by \cite{IS} for knot projections and denoted by $A^{-1}$ following \cite{IS}.   In this paper, we also call the local replacement $A^{-1}$ {\it{non-Seifert resolution}}.  
\begin{definition}[circle arrangements and circle numbers]
For a knot projection $P$, we apply $A^{-1}$ at every double point, and then we have a collection of circles having no double point on the sphere.  The collection of circles resulting from a knot projection $P$ is denoted by $\tau(P)$.  
Let us call this $\tau(P)$ {\it{circle arrangements}} (e.g.~Fig.~\ref{6t}).    
The number of circles in $\tau(P)$ is called a {\it{circle number}} and denoted by $|\tau(P)|$ (e.g.~ $|\tau(P)|$ $=$ $3$ in Fig.~\ref{6t}).  All together, the notions of circle numbers and circle arrangements are called {\it{circle invariants}}.  Note that circle invariants can be defined on both oriented and unoriented knot projections.    
\end{definition}
\begin{figure}[htbp]
\includegraphics[width=5cm]{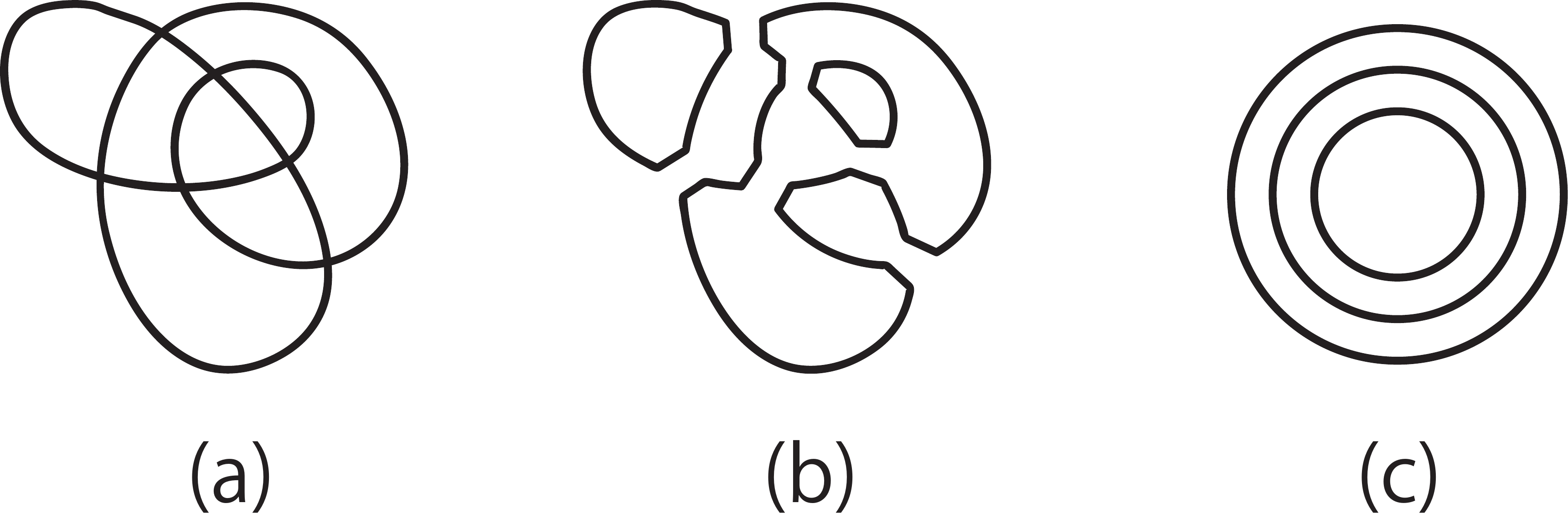}
\caption{(a) a knot projection $P$, (b) $\tau(P)$, (c) the same $\tau(P)$ as (b) under sphere isotopy.}\label{6t}
\end{figure}
\begin{theorem}
Let $P$ be an arbitrary knot projection.  $\tau(P)$ is invariant under strong (1, 2) homotopy.  
\end{theorem}
\begin{proof}
Fig.~\ref{7t} proves the claim.  
\begin{figure}[h!]
\includegraphics[width=5cm]{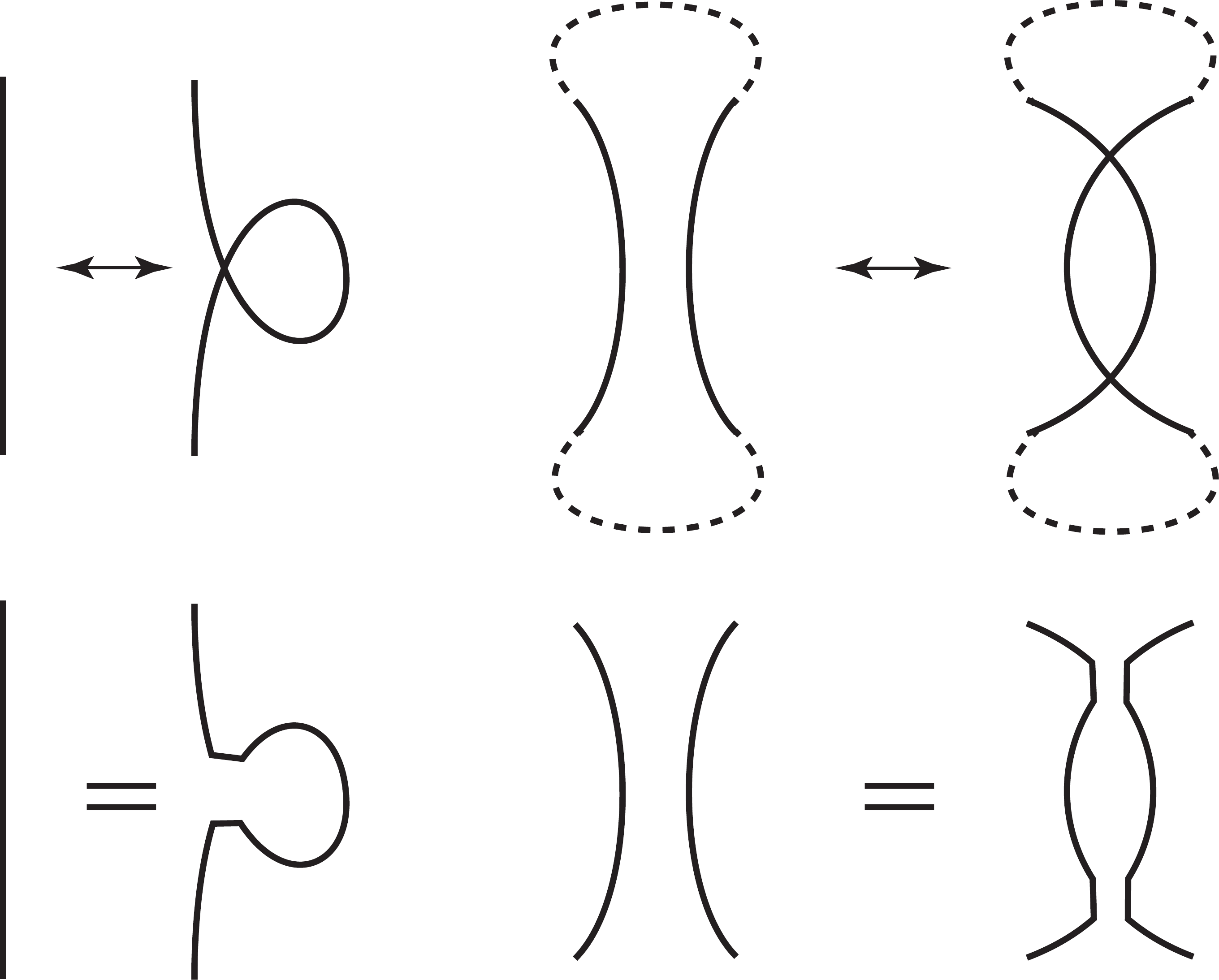}
\caption{Non-Seifert resolution $A^{-1}$ for 
RI and strong RI\!I.}\label{7t}
\end{figure}
\end{proof}
\begin{corollary}
Let $P$ be an arbitrary knot projection.  
$|\tau(P)|$ is invariant under strong (1, 2) homotopy.  
\end{corollary}
\begin{proposition}
Let $P$ be an arbitrary knot projection.  $\tau(P)$ is strictly stronger than $|\tau(P)|$.  
\end{proposition}
\begin{proof}
It is easy to see the canonical map $\tau(P) \mapsto |\tau(P)|$ by counting the number of circles in $\tau(P)$.  In Fig.~\ref{yhyou}, $|\tau(3_1)|$ $=$ $|\tau(6_3)|$ $=$ $3$ but $\tau(3_1)$ $\neq$ $\tau(6_3)$.  
\end{proof}
\begin{remark}
There exists two knot projections with the same circle arrangement which are not strongly (1, 2) homotopic.  See $7_C$ and $7_5$ (or $7_3$).    
\end{remark}
\section{Properties of circle numbers}\label{sec4}
The connected sum of two knot projections $P$ and $P'$ are defined by Fig.~\ref{connected}.
\begin{figure}[h!]
\includegraphics[width=8cm]{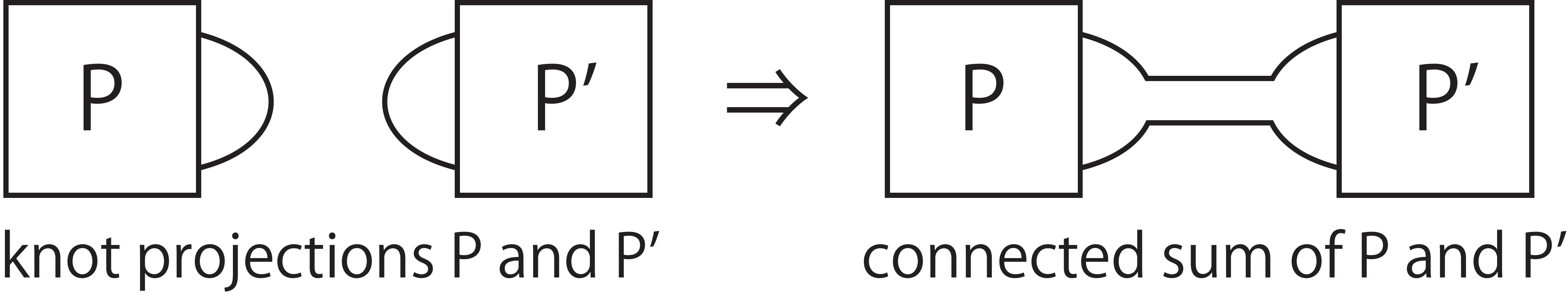}
\caption{Connected sum of two knot projections $P$ and $P'$.}\label{connected}
\end{figure}
\begin{theorem}\label{main4}
Let $P$ be an arbitrary knot projection.  
$|\tau(P)|$ has the following properties.  
\begin{enumerate}
\item $|\tau(P)|$ is an odd integer.  \label{c1}
\item $|\tau(P \sharp P')|$ $=$ $|\tau(P)|$ $+$ $|\tau(P')|$ $-$ $1$ where $P \sharp P'$ is a connected sum of $P$ and $P'$.  \label{c2}
\item For a knot projection $P$, we give $P$ an arbitrary orientation.  If $P$ contains an element of the following list (Fig.~\ref{yFig2.4}) as a sub-diagram, $|\tau(P)| \ge 3$.  \label{c3}
\begin{figure}[h!]
\includegraphics[width=8cm]{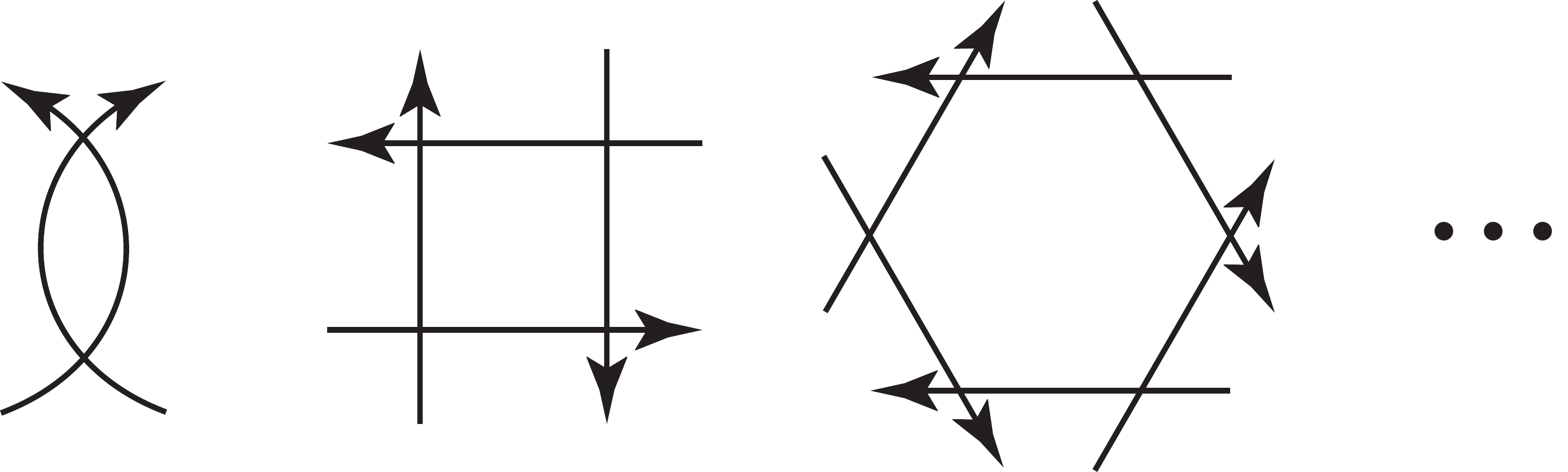}
\caption{List of $2n$-gons.}\label{yFig2.4}
\end{figure}
\end{enumerate}
\end{theorem}
\begin{proof}
\begin{enumerate}
\item By the definition, $|\tau(O)|$ is $1$.  
It is well known that an arbitrary knot projection and the simple circle $O$ can be related by a finite sequence of the first, second, and third flat Reidemeister moves, where the first and second flat Reidemeister moves are defined by Fig.~\ref{y4t} and the third flat Reidemeister move is defined by Fig.~\ref{y11t}.  
\begin{figure}[htbp]
\includegraphics[width=5cm]{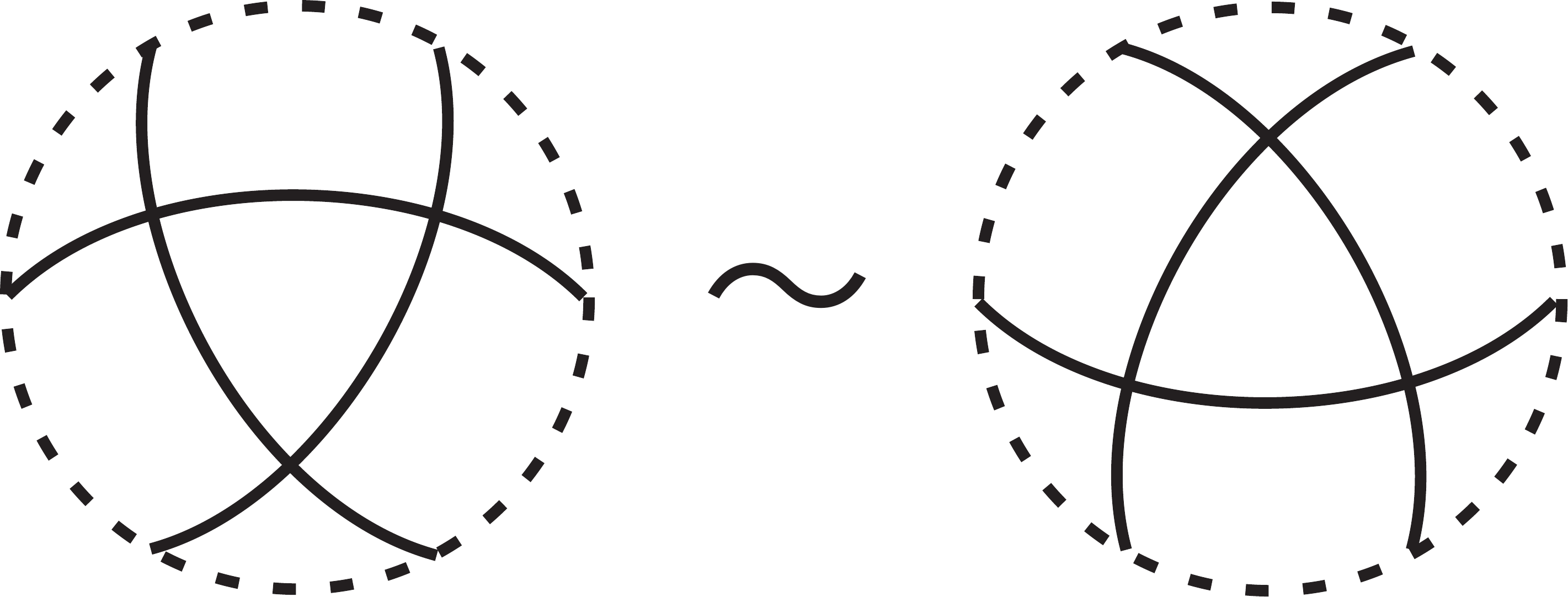}
\caption{The third flat Reidemeister move.}\label{y11t}
\end{figure}
We have already shown that $|\tau(P)|$ is invariant under 
RI and strong RI\!I.  Thus, it is sufficient to show that $|\tau(P)|$ $\equiv$ $0$ (mod $2$) under 
weak RI\!I and the third flat Reidemeister move.  

First, let us consider 
weak RI\!I by looking at Fig.~\ref{yFig2.2}.  For a knot projection $P$ with an orientation (the upper-left of Fig.~\ref{yFig2.2}), we have $\tau(P)$ appearing as in the left figure of the second row.  We consider the  possibilities of the connections of arcs of 
finitely many simple circles 
which we focus on. We can draw them as locally two types of $\tau(P)$, as in the upper right of the first and second rows in Fig.~\ref{yFig2.2}.  These two figures imply that $|\tau(P)|-|\tau(P')|$ is either $\pm 2$ or $0$ if $P$ and $P'$ is related by a single 
weak RI\!I.    
\begin{figure}
\includegraphics[width=8cm]{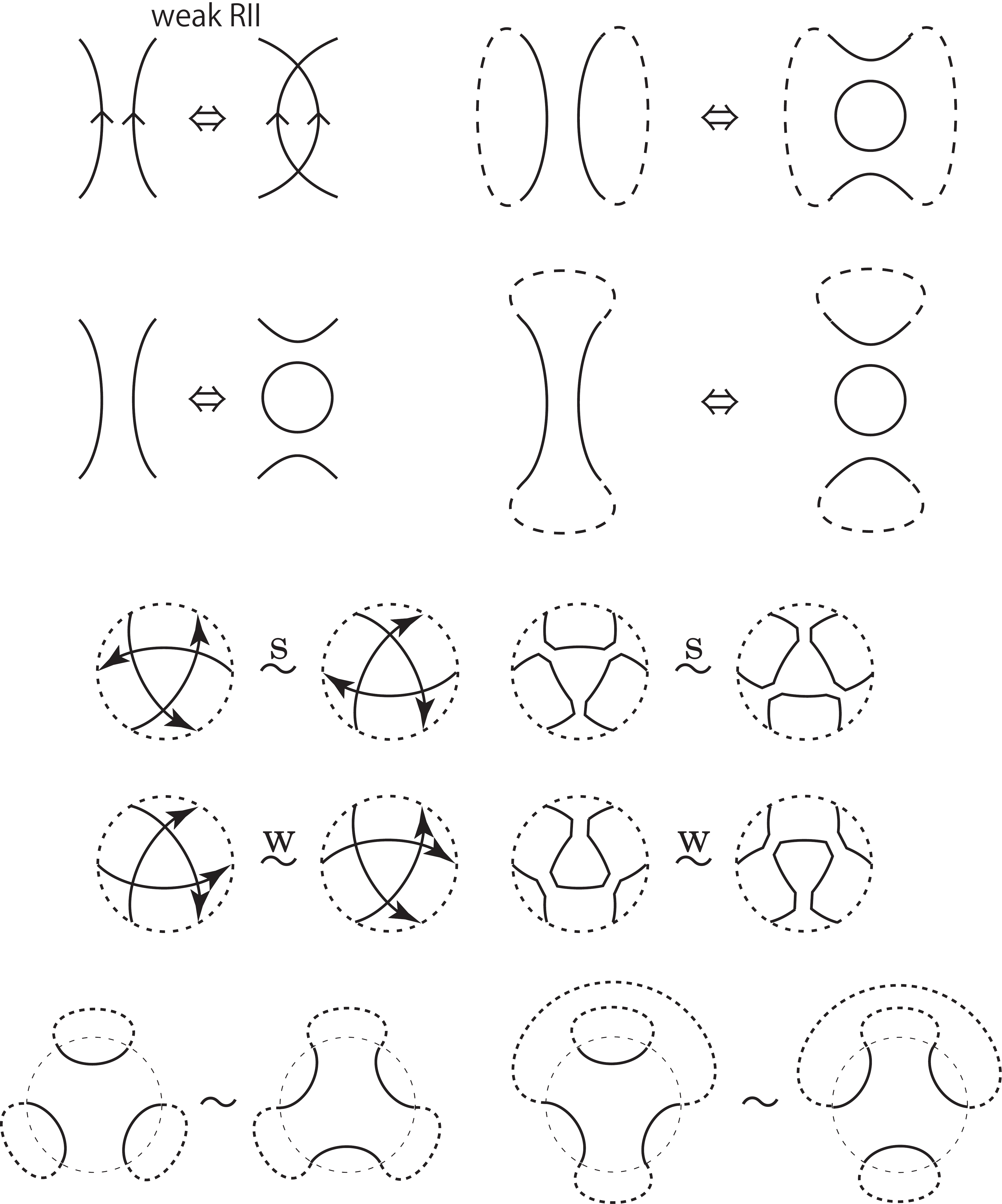}
\caption{The symbol $\Leftrightarrow$ in the figure (upper part) stands for 
weak RI\!I.  In the lower part of the figure, the symbol $\stackrel{w}{\sim}$ (resp.~$\stackrel{s}{\sim}$) represents the weak (resp.~strong) RI\!I\!I.}\label{yFig2.2}
\end{figure}

Next, we consider the third flat Reidemeister move.  The third flat Reidemeister move can be split into the two kinds of moves by the way of connection of three branches as in Fig.~\ref{y11t} (see Fig.~\ref{y12t}).  
\begin{figure}[htbp]
\includegraphics[width=5cm]{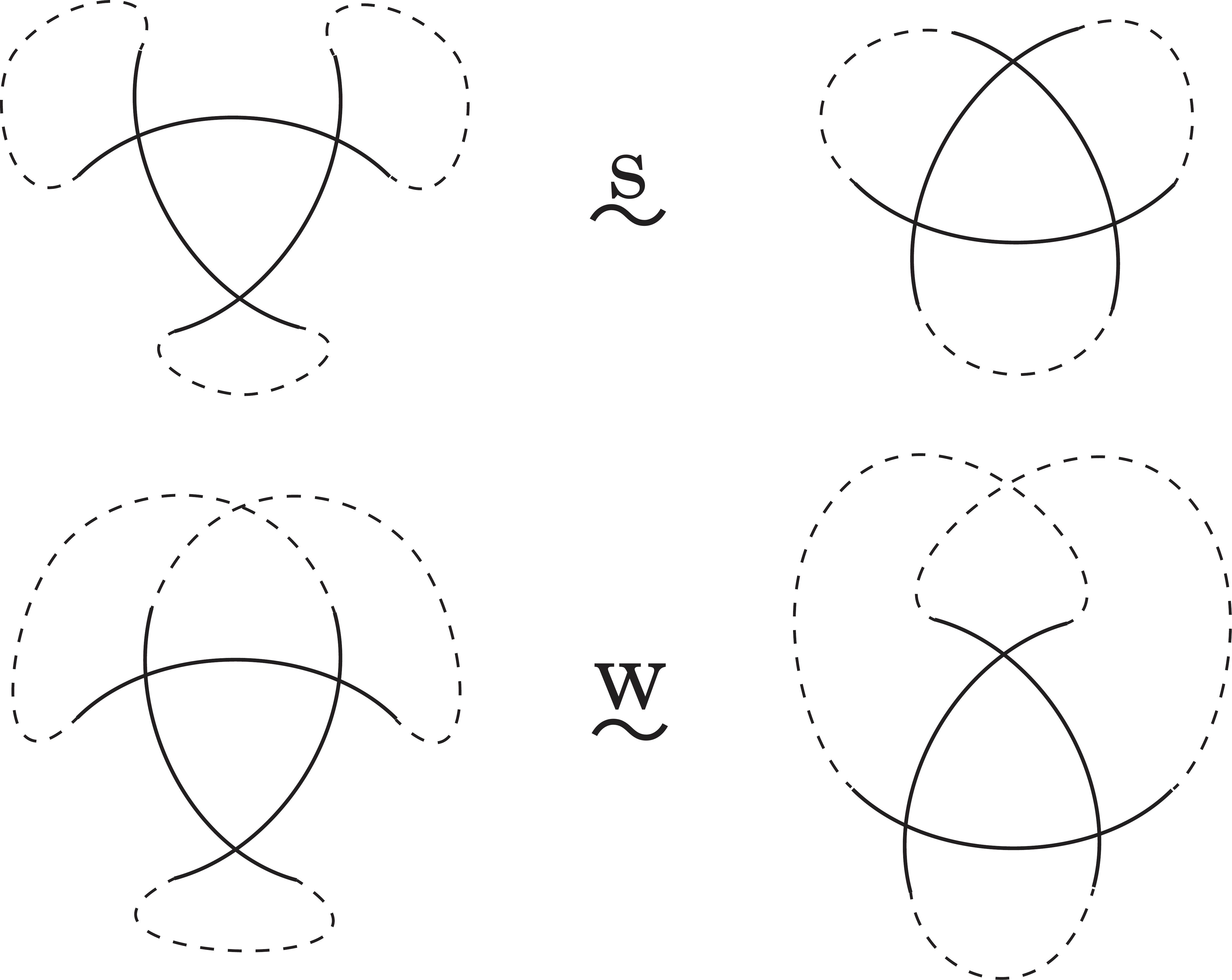}
\caption{The strong (upper) and weak (lower) 
RI\!I\!I.}\label{y12t}
\end{figure}
The upper (resp.~lower) local move of Fig.~\ref{y12t} is called the strong (resp.~weak) 
RI\!I\!I.  

These two local moves can be detected easily by giving an orientation to a knot projection.  In fact, one appears in the left of the third line and the other appears in the fourth line of Fig.~\ref{yFig2.2} if any orientations of knot projections are given.  

We apply non-Seifert resolutions to the two types of the third flat Reidemeister move (the left column in the third and fourth lines of Fig.~\ref{yFig2.2}).  Therefore, it is sufficient to consider to the two cases of connections of three arcs that are parts of simple circles (the bottom line of Fig.~\ref{yFig2.2}).  The two figures in the bottom line of Fig.~\ref{yFig2.2} shows that $|\tau(P)| - |\tau(P')|$ is either $0$ or $\pm 2$ if a knot projection $P$ is related to the other knot projection $P'$ by a single third flat Reidemeister move.  This completes the proof of (\ref{c1}).  
\item For two knot projections $P_1$ and $P_2$, we give them any orientations, and denote by $P_1^{o}$ and $P_2^{o}$ these oriented knot projections.  The knot projection $P_2$ with the opposite orientation is denoted by $P_2^{\bar{o}}$.  Now we consider a connected sum $P_1 \sharp P_2$ of $P_1$ and $P_2$.  We can preserve the orientation of $P_1^{o}$ under this connected sum by choosing either $P_2^{o}$ or $P_2^{\bar{o}}$.  Note that the $\tau(P_2^{o})$ $=$ $\tau(P_2^{\bar{o}})$.  Apply non-Seifert resolution at all double points $P_1 \sharp P_2$.  The way of smoothing them is the same as those for $P_1$ and $P_2$ since we consider either $P_1^{o} \sharp P_2^{o}$ or $P_1^{o} \sharp P_2^{\bar{o}}$.  
\item If there exists at least one part as one of the list, we have at least two circles in $\tau(P)$.  By (\ref{c1}), we have $|\tau(P)| \ge 3$.  
\end{enumerate}
\end{proof}
\begin{remark}
Let $P$ be a 
knot projection and set $a(P)$ $=$ $-\frac{1}{2}(J^+ + 2St)$ $=$ $- J^{+}/2 - St$, where $J^+$ and $St$ are Arnold invariants of immersed plane curves when we assign $\infty$ to one arbitrary chosen region.  It is easy to see that $a(P)$ does not depend on the choice of $\infty$ and is an invariant of 
knot projections \cite[Page 997, Corollary 2]{polyak}.  Moreover, it is easy to note that $a(P)$ is invariant under strong (1, 2) homotopy.  

We can also demonstrate that $|\tau(P)|$ is independent of $a(P)$.  For instance, let $O$ (resp.~$3_1$) be the simple circle (resp.~the trefoil projection) and let $5_1$ (resp.~$6_3$) be the $5_1$ (resp.~$6_3$) projection as shown in Fig.~\ref{yhyou}.  On one hand, $a(O)=0$, $a(3_1)=-1$, and $a(5_1)=a(6_3)=-2$.  On the other hand, $|\tau(O)|=1$, $|\tau(3_1)|=|\tau(6_3)|=3$, and $|\tau(5_1)|=5$.  As previously demonstrated, circle arrangement is strictly stronger than circle number, and as shown in Fig.~\ref{yhyou}, $\tau(O)$, $\tau(3_1)$, $\tau(5_1)$, and $\tau(6_3)$ are all mutually different.  
\end{remark}
\section{Proof of Theorem \ref{main3}}\label{sec4.3}
In this section, we prove Theorem \ref{main3}.  To avoid confusion, recall that the symbol $\sim_s$ (resp.~$\sim_w$) represents strong (1, 2) homotopy (resp.~weak (1, 2) homotopy) (see the definition of $\sim_s$ and $\sim_w$ in the statement of Theorem \ref{main1}).   
\begin{proof}
The assumption is that there exists a knot projection $P$ such that $P$ is equivalent to the simple circle $O$ up to weak (1, 2) homotopy (i.e.~$P$ $\sim_w$ $O$) and $P$ is also equivalent to 
$O$ 
up to strong (1, 2) homotopy (i.e.~$P$ $\sim_s$ $O$).  The condition that $P$ $\sim_s$ $O$ implies that $|\tau(P)|$ $=$ $1$.  On the other hand, $P$ $\sim_w$ $O$ implies that there exists a finite sequence ($\ast$) of $1a$ and $w2a$ from $O$ to $P$ by Lemma \ref{lem1w}.  

Now, we assume that the sequence ($\ast$) contains at least one $w2a$ ($\star$).  Therefore, there exists a knot projection $P'$ with at least one incoherent $2$-gon such that $P' \stackrel{1}{\sim} P$.  Then, $|\tau(P')|$ $=$ $|\tau(P)|$ $=$ $1$.  By Theorem \ref{main4} (\ref{c3}), $|\tau(P')|$ $\neq$ $1$ is a contradiction.  Thus, the assumption ($\star$) is false and then the sequence ($\ast$) contains only $1a$.  Then $P$ $\stackrel{1}{\sim}$ $O$.

Conversely, assume that $P \stackrel{1}{\sim} O$.  Then, in particular, $P$ satisfies both $P \sim_w O$ and $P \sim_s O$.  This completes the proof of Theorem \ref{main3}.  
\end{proof}
\section{Properties of circle arrangements}\label{sec5}
We call $P$ a {\it{prime}} knot projection if $P$ is non-trivial and cannot be a connected sum of non-trivial knot projections.
\begin{theorem}\label{main5}
For any circle arrangement $s$, there exists a knot projection $P$ such that $\tau(P)=s$.  In particular, for any positive integer $m$, there exists a knot projection $P$ such that $|\tau(P)|=m$.  

Moreover, we can choose a prime knot projection as $P$.  
\end{theorem}
After we demonstrate Lemmas \ref{l1}, \ref{l2}, and \ref{l3}, we will prove Theorem \ref{main5}.  
\begin{lemma}\label{l1}
A circle arrangement $\tau(P)$ can be locally changed as in Fig.~\ref{yFig2.5} from $\tau(P)$ to $\tau(P \sharp P')$ by considering a  connected sum $P \sharp P'$ if $P'$ is the trefoil projection or $6_3$ projection where the trefoil projection and $6_3$ projection look as in Fig.~\ref{yFig2.5}.  
\begin{figure}[htbp]
\includegraphics[width=8cm]{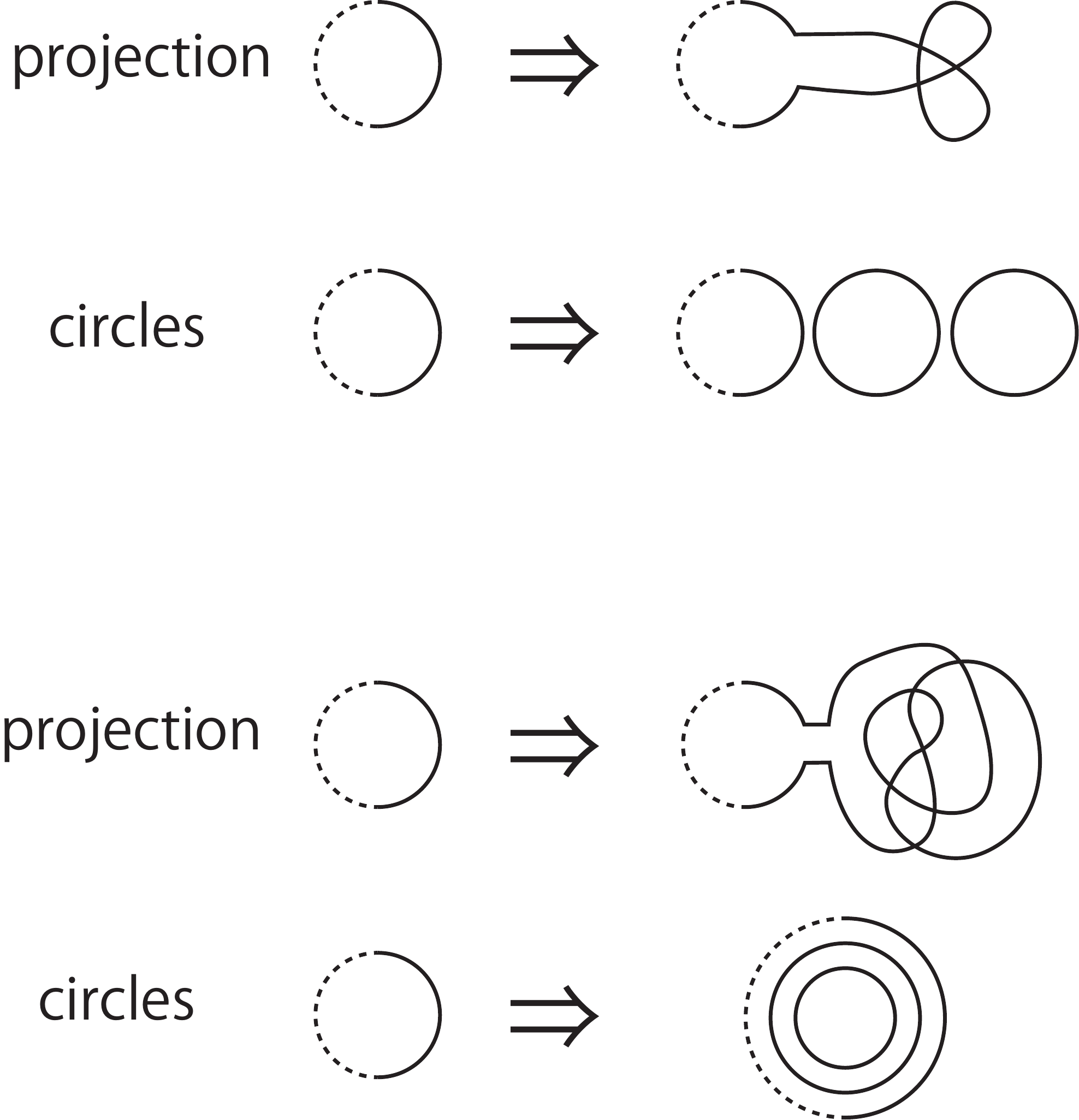}
\caption{The connected sums of a knot projection and either the trefoil projection (upper) or $6_3$ projection (lower).  In the second and fourth lines, we omit other circles which remain.}\label{yFig2.5}
\end{figure}
\end{lemma}
\begin{proof}
It is easy to see this proof by the definition of the circle arrangement.  
\end{proof}
\begin{lemma}\label{l2}
Let $P$ be a knot projection and assume $|\tau(P)| \ge 3$.  There exists a $2$-disk $D$ on the $2$-sphere such that $D$ contains exactly two circles that are separated or nested in $D$ as illustrated in Fig.~\ref{yFig2.6}.  
\begin{figure}[h!]
\includegraphics[width=5cm]{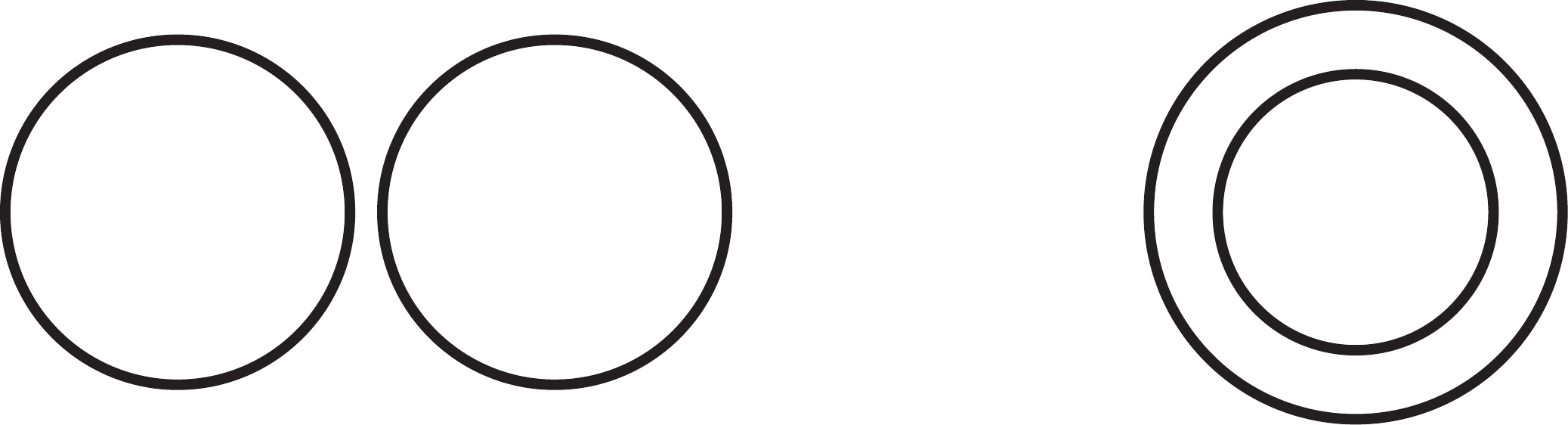}
\caption{Separated circles (left) and nested circles (right).}\label{yFig2.6}
\end{figure}
\end{lemma}
\begin{proof}
We pick up the innermost circle in a nested collection of circles and choose one of them arbitrarily, denoting it $C$.  If the circle $C$ is not encircled as in the right figure of Fig.~\ref{yFig2.6}, then there should exist at least one circle as in the left figure of Fig.~\ref{yFig2.6}.  Thus, the proof is completed.  
\end{proof}
\begin{lemma}\label{l3}
There exists a finite sequence generated by $1a$ and $s2a$ to obtain a prime knot projection from any non-prime knot projection.  
\end{lemma}
\begin{proof}
\begin{figure}[h!]
\includegraphics[width=5cm]{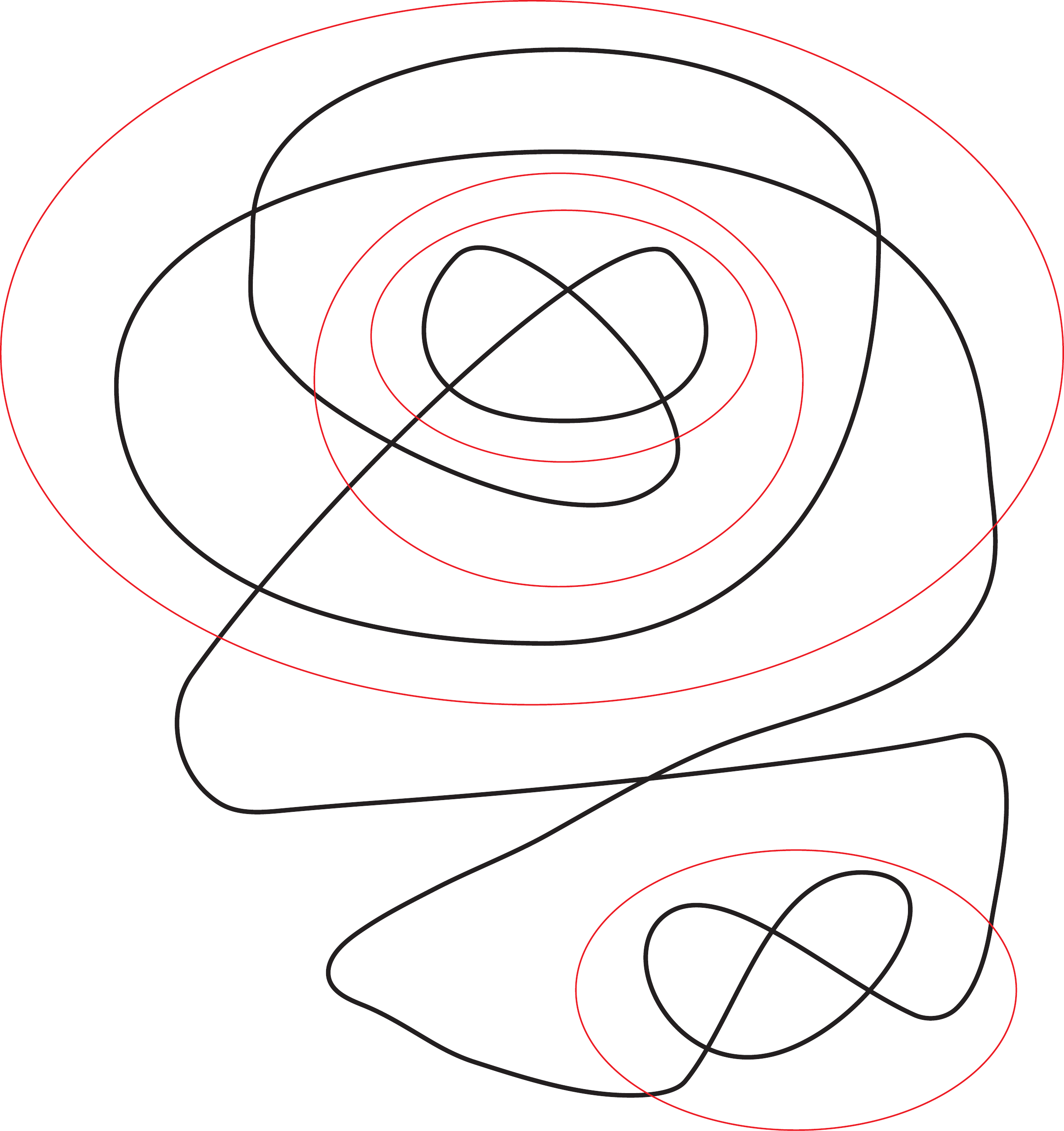}
\caption{A prime decomposition.}\label{y9t}
\end{figure}
If a knot projection $P$ is a simple  curve, it is easy to see that $4_1$ projection (cf.~Fig.~\ref{yhyou}) is obtained by $1a$ and $s2a$.  In the following, we assume that a knot projection is a non-trivial knot projection.  
We consider a prime decomposition as Fig.~\ref{y9t}. 
If a knot projection $P$ has a non-trivial sub-curve such that the sub-curve has exactly two end points and if we can find a simple arc closing the two endpoints of the curve that produces a prime knot projection and where the closing does not create double points, then the non-trivial sub-curve is called a {\it{prime part}} of $P$.
 Let us encircle the prime part at one of the most inner regions by one circle, called a {\it{red circle}}, where we arbitrarily select an  innermost prime part.  If we replace the prime part 
 with a simple arc, we precede to the next stage and repeat; we encircle the prime part of one of the innermost regions by a red circle that must not intersect the other red circles.  We continue to draw red circles not intersecting the other red circles.  Finally, we draw a circle whose part is contained by the outermost region and select a next-innermost prime part and repeat the above process.  For each step, if the outermost red circle is already given by the former process, we omit drawing it
 (e.g.,~Fig.~\ref{y9t}).  Since the number of double points in a knot projection is finite, every  encircling must eventually stop.  As a result, we can choose a decomposition where each red circle $r$ contains a 
prime part where we replace 
prime parts $t_1, t_2, \dots, t_m$ with simple arcs $a_1, a_2, \dots, a_m$ if the other red circles $r_1, r_2, \dots, r_m$ are contained in $r$.  
 
\begin{figure}[h!]
\includegraphics[width=5cm]{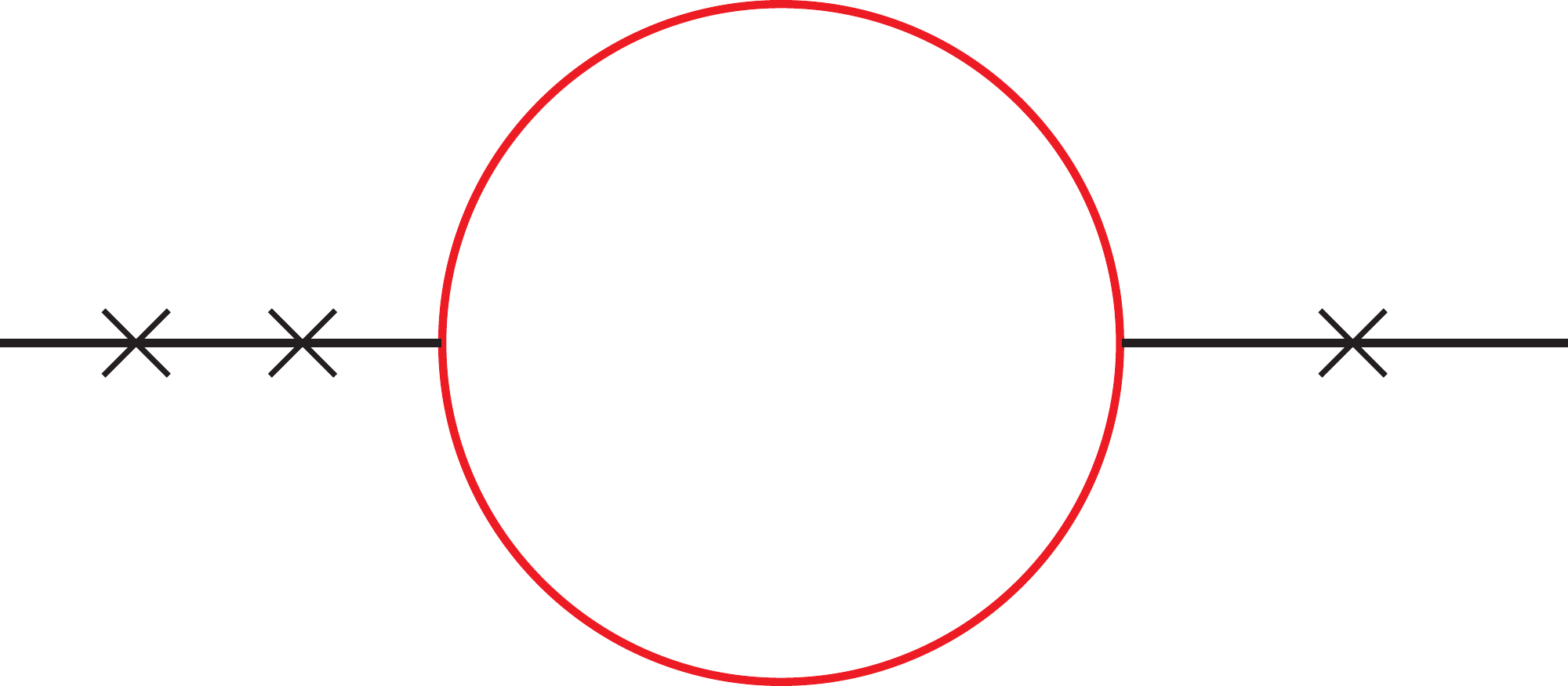}
\caption{Marked arcs and a circle encircling a prime part.}\label{y10t}
\end{figure}
Let us look at Fig.~\ref{y10t}.  Every red circle can be drawn as in Fig.~\ref{y10t}.  Call the two arcs of Fig.~\ref{y10t} {\it{marked arcs}}.  Every arc of the sub-diagram within the red circle of Fig.~\ref{y10t} is called a {\it{prime part's arc}}.  Choose two non-marked arcs $\alpha_1$ and $\alpha_2$ in the same region where one is picked from a prime part's arc $\alpha_1$ and the other is picked from a distinct neighboring prime part's arc $\alpha_2$.  Apply the operation shown in Fig.~\ref{yFig2.7} to arcs $\alpha_1$ and $\alpha_2$.  This operation is one $s2a$ or a pair of $s2a$ and $1a$; if two arcs $\alpha_1$ and $\alpha_2$ cannot make a coherent $2$-gon directly, we make the allowable situation by applying one $1a$ (see Fig.~\ref{yFig2.7}).  
\begin{figure}[htbp]
\includegraphics[width=8cm]{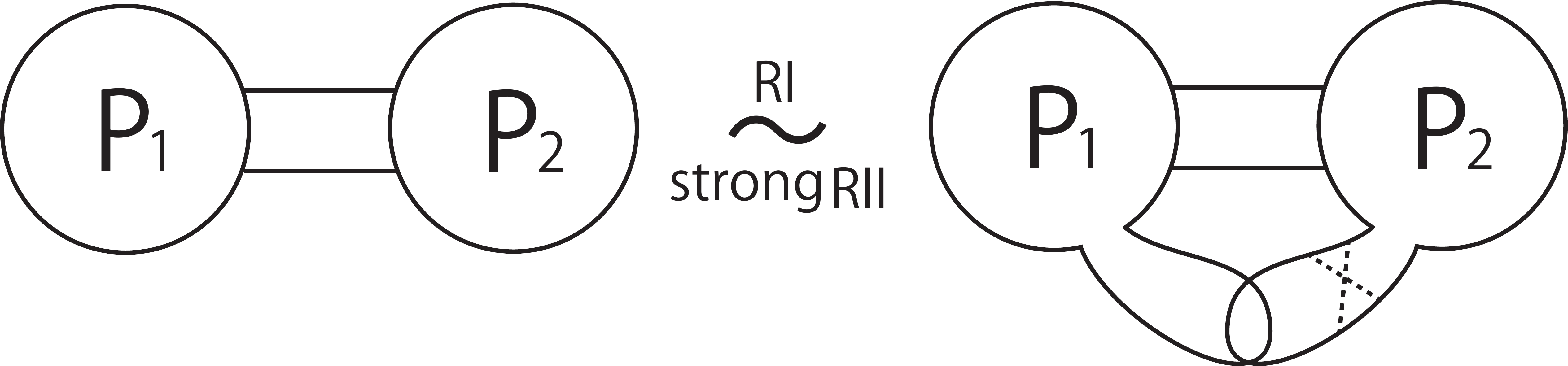}
\caption{Operation $s2a$ or a pair of $s2a$ and $1a$ to make a prime knot projection from a connected sum of knot projections.  Dotted arcs represent the possible need for $1a$ before applying $s2a$.}\label{yFig2.7}
\end{figure}
As a result, the relation of the two neighboring prime parts is not a connected sum and so we should delete the red circle between the two prime parts.  We apply this operation to every red circle, and so we can eventually resolve all the connected summations to obtain a prime knot projection.  Thus, the proof is completed.  
\end{proof}

Next, we prove Theorem \ref{main5}.  
\begin{proof}
If $|\tau(P)| \le 2$, $|\tau(P)|=1$ by Theorem~(\ref{c1}).  We have $|\tau(4_1)|=1$ (the symbol $4_1$ presents a knot projection, see Fig.~\ref{yhyou}).  Thus, we can assume that $|\tau(P)| \ge 3$ in the following.  
Let $Q_1$ (resp.~$Q_2$) be the left (resp.~right) figure of Fig.~\ref{yFig2.6}.  By Lemma \ref{l2}, we can find at least one copy of either $Q_1$ or $Q_2$ within any arbitrary circle arrangement.  For such an arbitrary circle arrangement $s=s_1$, obtain $s_2$ by deleting this part, denoted by $R_1$ $=$ $Q_1$ or $Q_2$.  In $s_2$, we can find at least one copy of $Q_1$ or $Q_2$ which we denoted by $R_2$.  Repeating the deletions, we eventually have exactly one circle $s_{m+1}$ and a sequence $R_1 R_2 R_3 \dots R_{m-1} R_m$ for some positive integer $m$.  

Next, we will construct any circle arrangement from one circle $O$.  For one circle $O$, we put $R_m$ in the two regions given by $O$ on $S^2$.  Using Lemma \ref{l1}, if $R_m$ $=$ $Q_1$ (resp.~$Q_2$), we consider the connected sum of the trefoil (resp.~$6_3$) projection and $O$ as in Fig.~\ref{yFig2.5}.  Next, for $R_{m-1}$, we consider an appropriate connected sum by specifying $R_{m-1}$ and $s_{m-1}$.  By applying connected sums step by step as above, we finally obtain a knot projection $P$ such that $\tau(P)$ $=$ $s_1$ $=$ $s$ and $P$ consists of finitely many trefoil and $6_3$ projections.  

Next we consider the goal of obtaining a prime knot projection such that $\tau(P)$ $=$ $s$.  To obtain such a knot projection, every time we consider a connected sum in the above step $R_k \to R_{k-1}$, we apply an operation $s2a$ or a pair $(1a)(s2a)$ as Fig.~\ref{yFig2.7} (e.g., Fig.~\ref{yFig2.8}).  
\begin{figure}[h!]
\includegraphics[width=8cm]{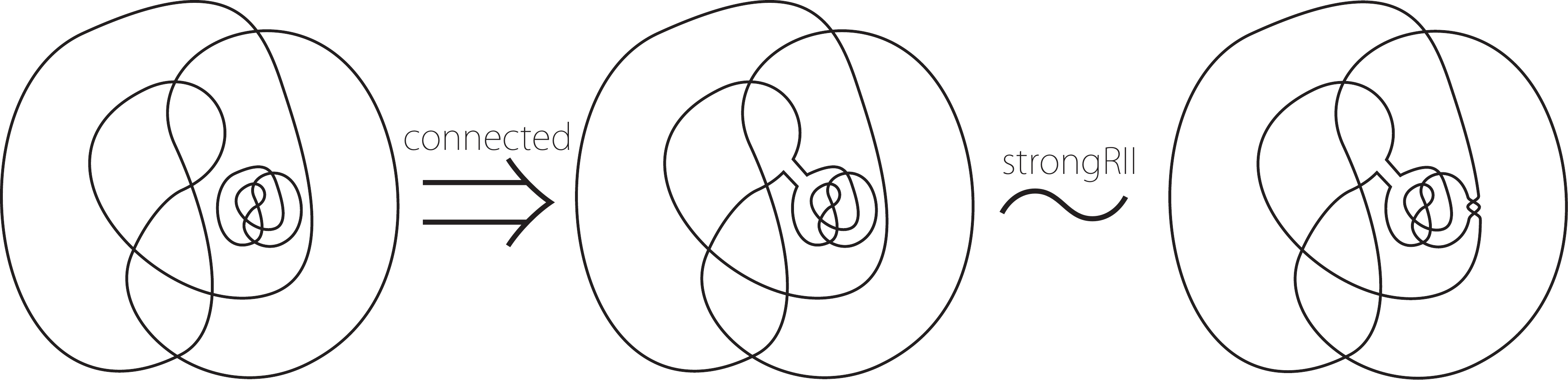}
\caption{Example of a prime knot projection from the connected sum of two $6_3$ projections.}\label{yFig2.8}
\end{figure}
\end{proof}
\section{Table of circle numbers and circle arrangements}\label{sec6}
Finally, we include a table of circle invariants on prime knot projections 
without $1$-gons up to seven double points (Fig.~\ref{yhyou}).  
The symbols for circle arrangements appearing in Fig.~\ref{yhyou} are defined by Fig.~\ref{yFig2.3}.  
Note that the following two points.  
The first point explains the meaning of the symbols $7_A$, $7_B$, and $7_C$.  Traditionally, $c_m$ denotes the projection image of the knot $c_m$ in Rolfsen table.  
To tabulate prime knot projections without $1$-gons up to seven double points, all knot projections are obtained from $c_m$ by flypes of knot projections defined by Fig.~\ref{fly} (cf.~Tait Conjecture).  
\begin{figure}[h!]
\includegraphics[width=5cm]{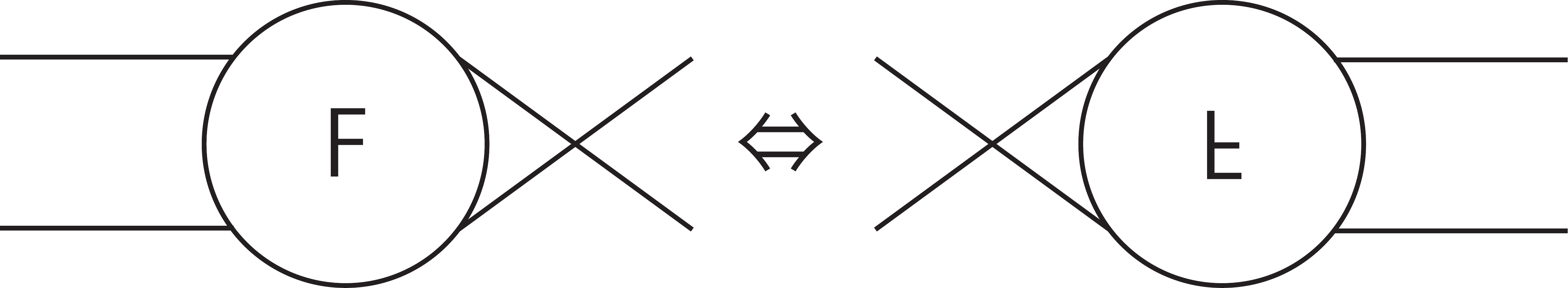}
\caption{A flype.}\label{fly}
\end{figure}
Thus, exactly three knot projections should be added, i.e., $7_A$ (from $7_6$), $7_B$ (from $7_7$), and $7_C$ (from $7_5$).  The second point is that we obtain the definitions of lines with $w$ and $s$ in Fig.~\ref{yhyou}.  Two knot projections are connected by a line with $s$ (resp.~$w$) if a finite sequence generated by moves of type $1a$ and a single $s2a$ (resp.~$w2a$) that connects the two knot projections is found.  

\begin{figure}[htbp]
\includegraphics[width=8cm]{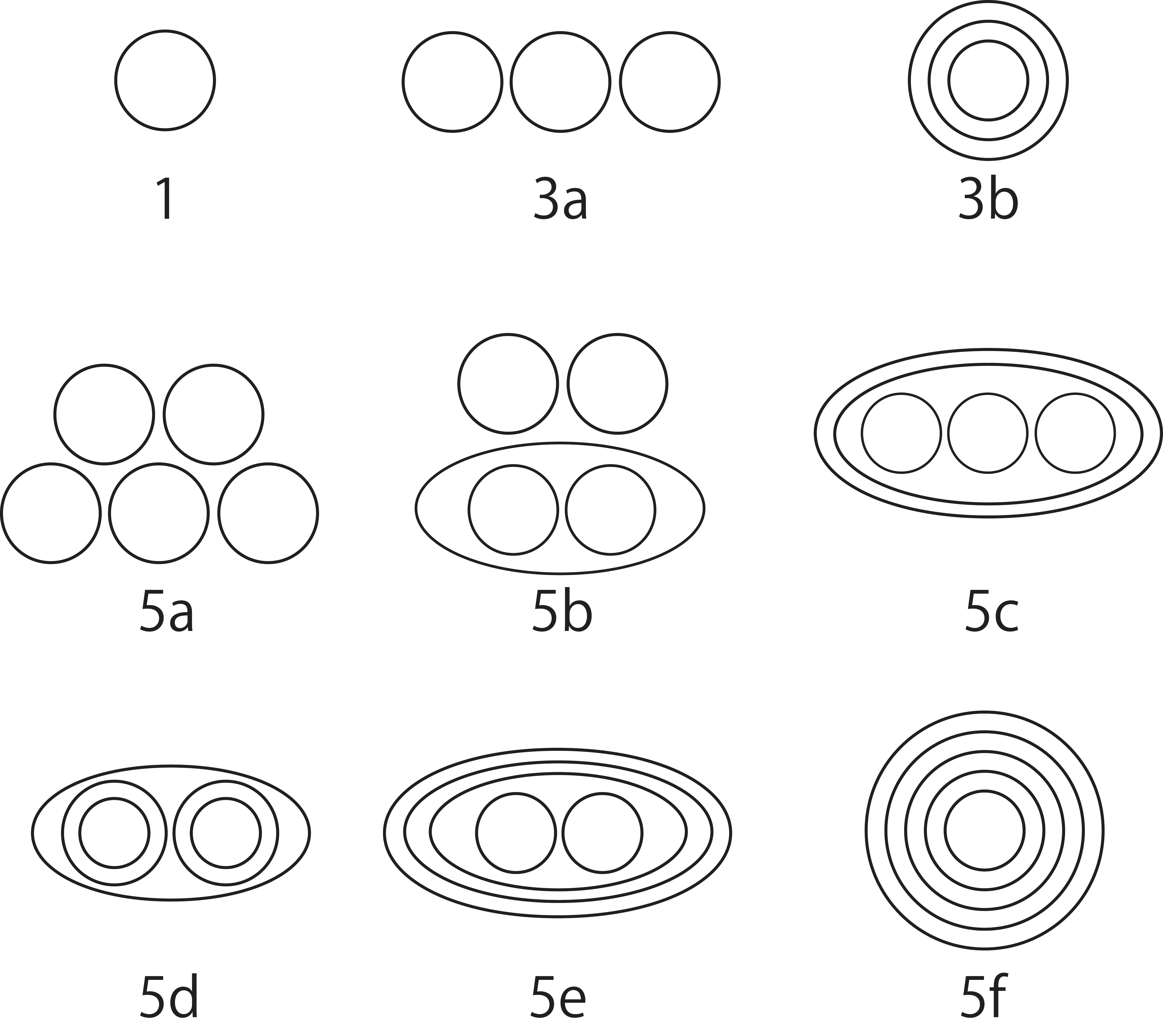}
\caption{Definition of symbols assigned to circle arrangements for up to five circles on a sphere.}\label{yFig2.3}
\end{figure}

\begin{figure}[htbp]
\includegraphics[width=12cm]{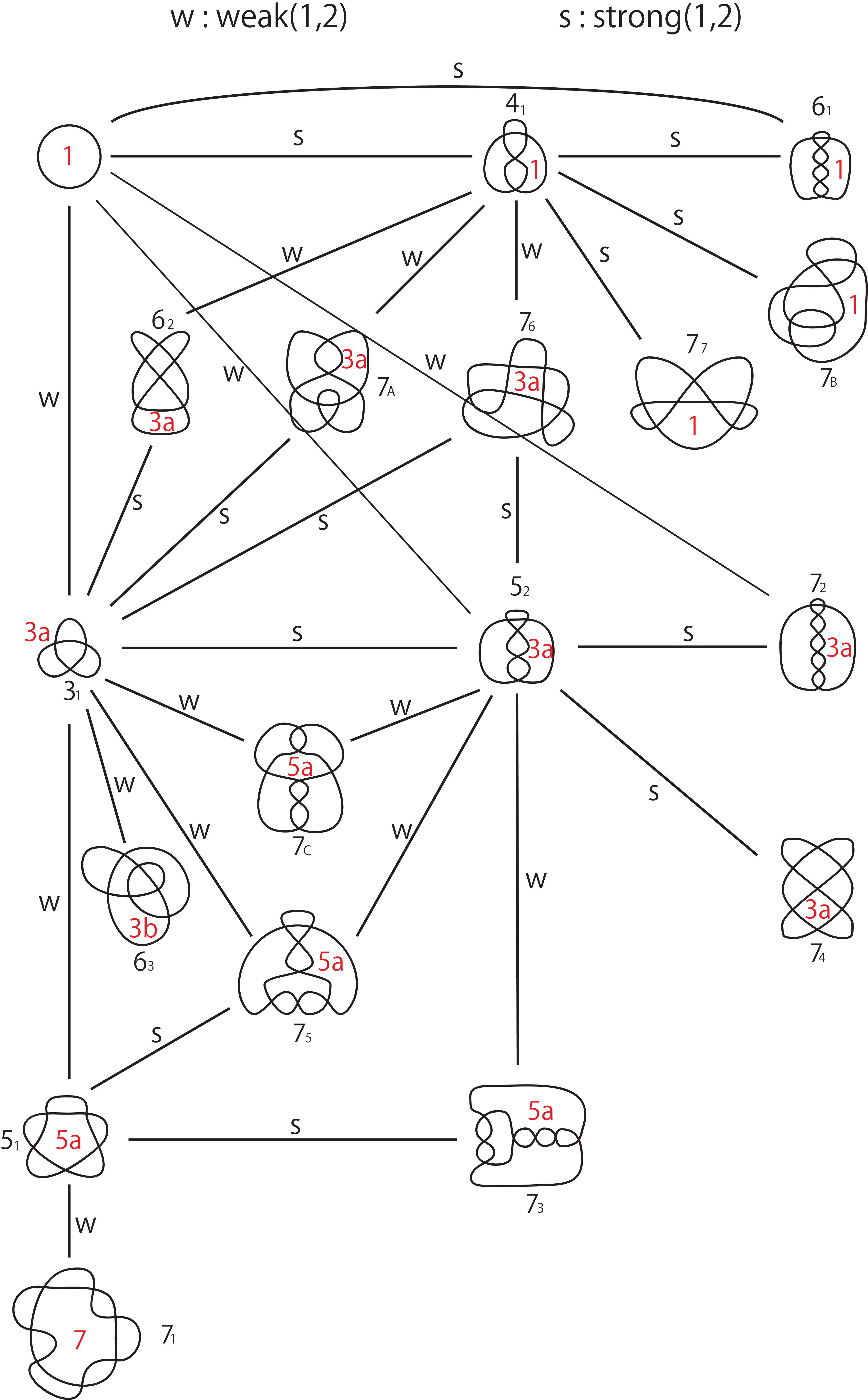}
\caption{Prime knot projections 
without $1$-gons
 for small numbers of double points with circle arrangements.  Numbers in this table represent circle numbers (e.g., $3$ of the symbol $3a$).}\label{yhyou}
\end{figure}

\section*{Acknowledgements}
The authors would like to thank Professor Kouki Taniyama for his fruitful comments.  The authors would also like to thank the referee for his/her comments on an earlier version of this paper.  
The part of this work of N.~Ito was supported by a Waseda University Grant for Special Research Projects (Project number: 2014K-6292) and the JSPS Japanese-German Graduate Externship.  N.~Ito is a project researcher of Grant-in-Aid for Scientific Research (S) 24224002 (April 2016--).

\end{document}